\RequirePackage{fix-cm}
\documentclass[smallextended]{svjour3}       
\smartqed  

\usepackage{comment}

\usepackage[numbers]{natbib}
\usepackage{listings} 
\usepackage{fullpage,enumerate, enumitem}

\usepackage{url} 
\usepackage{float} 
\usepackage[colorlinks]{hyperref} 
\hypersetup{citecolor=blue}

\usepackage{multicol, blindtext}
\usepackage{longtable} 
\usepackage{setspace}
\usepackage[margin=2.5cm]{geometry}
\usepackage{float}

\usepackage{xcolor}

\usepackage{rotating}

\usepackage{array}

\usepackage{float}

\usepackage{graphicx}

\usepackage[font={small,it}]{caption}

\usepackage[T1]{fontenc}

\usepackage[english]{babel}
\usepackage{german}
\usepackage{ngerman}
\usepackage{makeidx}					

\usepackage{graphicx}					

\usepackage{textcomp}
\usepackage{amssymb}
\usepackage{amsmath}

\usepackage{float}						
\usepackage{rotating}					
\newcommand{\one}{\mathsf{\textbf{1}}}

\begin{document}

\title{Free CIR Processes
}


\author{Holger Fink         \and
        Henry Port$^{*}$   \and  
        Georg Schl\"{u}chtermann 
}


\institute{Holger Fink \at        
              Department of Computer Science and Mathematics, Munich University of Applied Sciences, Munich, Germany;
              Center for Quantitative Risk Analysis CEQURA, Ludwig-Maximilians-Universit\"at M\"unchen, Akademiestr. 1/I, 80799 Munich, Germany\\
              \email{holger.fink@hm.edu}           
           \and
           Henry Port$^{*}$ \at
              Chair of Financial Econometrics, Institute of Statistics, Ludwig-Maximilians-Universit\"at M\"unchen, Akademiestr. 1/I, 80799 Munich, Germany;
              Center for Quantitative Risk Analysis CEQURA, Ludwig-Maximilians-Universit\"at M\"unchen, Akademiestr. 1/I, 80799 Munich, Germany\\
              \email{henry.port@stat.uni-muenchen.de}\\
              {*} Corresponding author
           \and 
           Georg Schl\"{u}chtermann \at 
           Department of Mechanical, Automotive and Aeronautical  Engineering, Munich University of Applied Sciences, Munich, Germany;
           Faculty of Mathematics, Computer Science and Statistics, Ludwig-Maximilians-Universit\"at M\"unchen, Theresienstrasse 39/I, 80333 Munich, Germany\\ \email{georg.schluechtermann@hm.edu}
              }

\date{Received: date / Accepted: date}

\maketitle
\begin{abstract}
For stochastic processes of non-commuting random variables we formulate a Cox-Ingersoll-Ross (CIR) stochastic differential equation in the context of free probability theory which was introduced by D. Voicelescu. By transforming the classical CIR equation and the Feller condition, 
which ensures the existence of a positive solution, into the free setting (in the sense of having a strictly positive spectrum), we show the global existence for a free CIR equation. The main challenge lies in the transition from a stochastic differential equation driven by a classical Brownian motion to a stochastic differential equation driven by the free analogue to the classical Brownian motion, the so-called free Brownian motion. 
\end{abstract}
\section{Introduction}
The Black-Scholes framework is considered as one of the benchmarks for modeling the price process $(X(t))_{t\geq0}$ of an underlying asset. Initially mentioned in \cite{black1973pricing}, it is based on the stochastic differential equation (SDE) $$\text{d} X(t) = \mu X(t) \text{d}t + \sigma X(t) \text{d} B_{1}(t),$$
for $t\in[0,\infty[,$ where $\mu\in\mathbb{R},\;\sigma>0$ and $(B_1(t))_{t\geq0}$ denotes a Brownian motion. 

One major drawback of this model is the fact that it does not account for certain common properties of financial data, known as stylized facts, in particular, volatility clustering and the so-called leverage effect. The first term refers to the fact that volatility exhibits a highly autocorrelated structure and the second to the negative correlation between volatility and returns frequently found in financial data (cf. \cite{pagan1996econometrics, mandelbrot1997variation,cont2001empirical}).

A possible solution is to model the variance separately as a time-dependent stochastic process that accounts for these effects and to ensure that it returns to an average value in finite time, which is referred to as mean-reversion, and stays positive, as well.

In 1993 \cite{heston1993closed} addressed these issues by allowing the variance to be modeled by the Cox, Ingersoll and Ross (CIR) process as developed by \cite{cir}, via the SDE $$\text{d} s(t)^2 = (a-b s(t)^2)\text{d}t + \sigma\sqrt{s(t)^2} \text{d}B_{2}(t),$$
for $t\in [0,\infty[,$ with $a,b,\sigma>0$  and $(B_{2}(t))_{t\geq0}$ another Brownian motion s.t. $\rho \text{d} t =\text{d} B_{1}(t) \text{d} B_{2}(t),$ with 
$-1\leq\rho\leq1.$ \cite{feller1951two} has shown that as long as the so-called Feller condition 
\begin{equation}\label{feller}
2a\geq\sigma^2 
\end{equation}
holds there exists a positive unique solution. 

Since its discovery the CIR equation has been a vivid field of research and underwent several modifications: 
In order to model potential long memory effects in volatility which were observed by some researchers (cf. \cite{baillie1996fractionally,bollerslev1996modeling}), \cite{comte1998long} implemented a fractional Brownian motion as the driving process.
The fractional CIR equation has been studied, among others e.g. by \cite{comte2012affine,schluchtermann2016note} and \cite{fink2018fractional}.

 Yet another modification is the implementation of a Hurst index of $H<\frac{1}{2}.$ The resulting so-called rough Heston models give a good mixture of a decent fit to historical data and implied volatility without touching upon the curse of dimensions (cf. \cite{alos2007short,gatheral2018volatility,el2018perfect,jaber2018multi,el2019characteristic}).

The task of modeling prices for two or more assets on the other hand poses additional challenges. Taking covariances into account necessitates a joint model.
A promising candidate is given by Wishart autoregression processes, which were introduced in \cite{bru1991wishart}. Since Wishart processes do not need additional constraints to ensure positive definiteness almost surely (cf. \cite{gourieroux2006continuous}), they are suited especially well for the role of a matrix-valued CIR process. 

Nevertheless, with an increasing number of assets the complexity of volatility models adequately describing these systems increases rapidly, demanding the usage of more variables and therefore the curse of dimension (cf. \cite{gourieroux2010derivative}) threatens a feasible application. 

Some matrix ensembles such as the GUE (Gaussian unitary ensemble) and the  Wishart random matrices (these are of the form $\frac{1}{N}X X^{*},$ where $X$ is $N\times M$ random matrix with independent Gaussian entries and $X^{*}$ denotes its adjoint) behave like so-called free random variables in (eigenvalue) distribution, when their size gets very large, which is referred to as ``asymptotic freeness'' (cf. \cite{mingo2017free}). The limiting eigenvalue distribution of the former is given by the semi-circular distribution and for the latter by the Marchenko-Pastur distribution (also known as ``free Poisson distribution'').  
Since these matrix ensembles behave like free random variables in high dimensions we may employ the tools of free probability theory, which allow for concepts like convolution and a pendant to the central-limit theorem (cf. \cite{nourdin2014central}) to adequately work with random matrices and operators in the probabilistic context.

We propose so-called free CIR (stochastic differential) equations in the context of Voicelescu's ``free probability theory'' (cf. ``Background and outlook'' in \cite{voiculescu2016free}), which is basically a non-commutative analogue to classical probability theory with a special rule to calculate joint moments, called freeness.


Free stochastic calculus first appeared in \cite{speicher1990new}.
This theory was further developed by \cite{kummerer1992stochastic,biane1997free} and \cite{biane1998stochastic}, where among other ground-laying definitions, the notion of free stochastic processes and a free Brownian motion were introduced. For an in-depth study of the free It\^{o} integral, we refer to \cite{anshelevich2002ito}.  
This framework allows for defining the notion of free SDEs. 
In particular, free stochastic processes form a vivid research area (cf. \cite{biane1998processes,biane2001free,barndorff2002self,fan2006self,gao2006free,gao2008free,an2015poisson}).

Some processes, such as the item of interest of this paper, the CIR process, arise naturally as the solution of SDEs.
A special class of free SDEs is studied in depth in \cite{kargin2011free}, where the author provided proofs for the existence and uniqueness of a local solution of free SDEs of the form 
\begin{equation*}
\text{d}X(t)=a(X(t))\text{d}t+\displaystyle\sum_{k=1}^m b_k (X(t))\text{d}W(t) c_k (X(t)),
\end{equation*} 
for $t\in[0,\infty[,$ where $(W(t))_{t\geq0}$ is a free Brownian motion and $a, b_k, c_k$ are locally operator Lipschitz functions. To recall: We call a function $f:\mathbb{R}\rightarrow\mathbb{C}$ locally operator Lipschitz, if it is locally bounded, measurable, and if we also have that for all $C>0,$ there is a constant $K(C) > 0,$ with the property that $\lVert f(X) - f(Y)\rVert \leq K(C) \lVert X - Y \rVert,$ provided that $X$ and $Y$ both are self-adjoint operators such that $\lVert X \rVert , \lVert Y \rVert < C,$ with $\lVert \cdot \rVert$ denoting the operator norm. 

Unfortunately due to the fact that the coefficient functions do not fulfill the necessary locally operator Lipschitz condition, the local existence theorem does not imply the existence of a solution to the non-classical free CIR equation of the form 
\begin{equation}\label{classfreecirneu}
\text{d}X(t)=\left(a(t)-b X(t)\right)\text{d}t+\frac{\sigma(t)}{2} \sqrt{X(t)} \text{d}W(t) + \text{d}W(t) \sqrt{X(t)}\frac{\sigma(t)}{2},\quad X(0)=X_0 \in\mathcal{A}_+,
\end{equation}
for $a(t),\sigma(t)\in\mathcal{A}_+ , b>0$ and $t\in[0,\infty[.$
We denote by $\mathcal{A}_+$ the self-adjoint elements of the von Neumann algebra $\mathcal{A}$ with a strictly positive spectrum. Those are representable by the square of a self-adjoint operator which gives meaning to the expression $\sqrt{A}.$
Therefore a solution to this equation needs to be bounded below by zero (in the sense of having a positive spectrum), just as in the scalar-valued case. We introduce and show the existence of the non-classical free CIR equation and derive an equivalent condition to the classical Feller condition. Note that for possibly time-dependent operators $a(t), b(t), \sigma(t)^2$ the Feller condition is to be understood in the sense of $2 a(t) - \sigma(t)^2$ having a non-negative spectrum. When referring to the Feller condition \eqref{feller} in the context of operators that is how it is meant to be interpreted. 
Our main theorem states the global existence of a positive solution to the non-classical free CIR equation \eqref{classfreecirneu}. Additionally, by a simple argument, we get the same for the classical free CIR equation of the form
$$\text{d}X(t)=\left(a(t)-b X(t)\right)\text{d}t+\sqrt[\leftroot{0}\uproot{1}4]{X(t)}\sqrt{\sigma(t)} \text{d}W(t)\sqrt{\sigma(t)}\sqrt[\leftroot{0}\uproot{1}4]{X(t)},\quad X(0)=X_0 \in\mathcal{A}_+,$$
for $a(t),\sigma(t)\in\mathcal{A}_+ , b>0$ and $t\in[0,\infty[,$
which has a different solution than the non-classical free CIR equation, due to the non-commutativity of the process. Note that both free CIR equations have to be formulated in this way to ensure the selfadjointness needed for the positivity of the solution. In contrast to the free (non-commutative) cases, in the scalar-valued case both equations would be equal to the classical form $\text{d}X(t)=\left(a(t)-b X(t)\right)\text{d}t+\sigma(t)\sqrt{X(t)} \text{d}W(t)$ . 

\section{Free Probability and Free Stochastic Calculus}\label{basics}


In the following we give a short introduction to the theory of free probability and free SDEs. 
For the theory of operator theory and in particular von Neumann algebras we refer the reader to the multivolume works ``Theory of operator algebras'' by Takesaki (cf. \cite{takesaki1979theory,takesaki2013theory,takesaki2003theory}) and \cite{murphy2014c}. For an introduction to free probability we refer to \cite{voiculescu1992free,nica2006lectures,voiculescu2016free} and \cite{mingo2017free}. 
In the following we will draw heavily on \cite{biane1998stochastic} and \cite{kargin2011free}.

\begin{definition}
Let the tupel $\left(\mathcal{A}, \varphi\right)$ be a non-commutative probability space (cf. \cite{nourdin2012selected}), that is $\mathcal{A}$ is a von Neumann algebra and $\varphi$ a faithful, normal and unital trace on $\mathcal{A}.$ The self-adjoint elements of $X\in \mathcal{A}$ are referred to as (free) random variables.
\end{definition}
The trace $\varphi$ induces the p-norms $\lVert\cdot\rVert_p$ by 
\begin{equation*}
\lVert X \rVert_{p}=\varphi\left(\lvert X\rvert^p\right)^{\frac{1}{p}},\quad 1\leq p \leq\infty,
\end{equation*}
where $\lVert\cdot\rVert:=\lVert\cdot\rVert_{\infty}$ is the operator norm.
For more information on non-commutative integration see e.g. \cite{FackKos1986} and \cite{Terp1981}.

We proceed by defining the notion of free independence or freeness, as introduced by Voiculescu.
\begin{definition}
Let $\mathcal{A}_1, \dots, \mathcal{A}_n$ be subalgebras of $\mathcal{A}$ and let the corresponding elements be denoted by $A_{i}\in\mathcal{A}_{i}.$ We call $\mathcal{A}_1, \dots, \mathcal{A}_n$ free if
\begin{equation*}
\varphi\left(A_{i(1)}\dots A_{i(m)}\right)=0
\end{equation*}
whenever 
$$\varphi\left(A_{i(s)}\right)=0\text{ and }i(s+1)\neq i(s)\quad\text{for each }s.$$
We call random variables free if the algebras they generate are free.
\end{definition}
In order to introduce free SDEs, a ``free'' notion of Brownian motion is necessary.
\begin{definition}
A free Brownian motion is a stochastic process $(W(t))_{t\geq0}$ of elements in a von Neumann algebra with the following three properties: 
$W(0)=0, $ the increments $W(t)-W(s)$ are free from $\mathcal{W}_s=\left<W(\tau)\;|\;\tau\leq s\right> $ (the von Neumann algebra generated by the $W(\tau)$ with $\tau\leq s$) for $t>s,$ and $W(t)-W(s)$ follows a semicircular distribution with expectation $\varphi\left(W(t)-W(s)\right)=0$ and variance $\varphi\left(\left(W(t)-W(s)\right)^2\right)=t-s$ (cf. \cite{nourdin2012selected}).
\end{definition}
We proceed by introducing an integral of the form 
\begin{equation*}
\mathcal{I}=\int_0^T a\left(X(s)\right)\text{d}W(s) b\left(X(s)\right),
\end{equation*}
where $a\left(X(s)\right)$ and $b\left(X(s)\right)$ are operator-valued functions of $X(s)$ and $T>0$. Certain properties on $a$ resp. $b$ will be specified later. For a detailed construction of this free stochastic integral, we refer to \cite{anshelevich2002ito} or \cite{biane1998stochastic}. We will describe the general construction shortly as it is done in \cite{kargin2011free}.


Given an interval $[0,T]$ and $s\in[0,T],$ 
we let $a_s,b_s\in\mathcal{W}_s$ and assume that  $s\mapsto a_s$ and $s\mapsto b_s$ are continuous maps w.r.t. to $\lVert\cdot\rVert.$ \color{black} Consider 
$s_0,\dots,s_n,\tau_1,\dots,\tau_n\in\mathbb{R}$ with $0=s_0\leq s_1\leq\dots\leq s_n=T$
and $0\leq\tau_k\leq s_{k-1}$.
We denote the collection of all $s_i$ and $\tau_j$ by $\Delta:=\{s_i,\tau_j\;|\;0\leq i\leq n\;,\;1\leq j\leq n\}.$
Consider the expression 
\begin{equation*}
\mathcal{I}\left(\Delta\right)=\displaystyle\sum_{i=1}^n a_{\tau_i}\left(W(s_i)-W(s_{i-1})\right)b_{\tau_i}.
\end{equation*}
For $d\left(\Delta\right)=\max_{1\leq k\leq n}(s_k-\tau_k)$ we get that
\begin{equation*}
\lim_{d\left(\Delta\right)\rightarrow 0} \mathcal{I}\left(\Delta\right)=\mathcal{I},
\end{equation*}
where the limit is meant w.r.t. $\lVert\cdot\rVert$ and is independent of the choice of $s_i$ and $\tau_i.$ 
$\mathcal{I}$ is called the ``free stochastic integral''.
The convergence relies heavily on the so-called free Burkholder-Gundy inequality (see \cite[Theorem~3.2.1]{biane1998stochastic}).
In the following lemma we state the free analogues of the It\^{o} formula in an abbreviated form as we will use later on. For a detailed discussion we refer to \cite{biane1998stochastic}. 
\begin{lemma}[Free It\^{o}]\label{L:freeIto}  Let $a_t , b_t , c_t , d_t $ be operator-valued functions  and $(W(t))_{t\geq0}$ as above. 
Then 
\begin{eqnarray*}
a_t \text{d}t \cdot b_t \text{d}t&=& a_t\text{d}t \cdot b_t \text{d}W(t) c_t= a_t \text{d}W(t)  b_t \cdot c_t \text{d}t=0\\
a_t \text{d}W(t) b_t \cdot c_t \text{d}W(t) d_t &=& \varphi\left(b_t c_t\right) a_t d_t\text{d}t
\end{eqnarray*}
\end{lemma}

From here on we will restrict ourselves to coefficients that do not depend on time explicitly (as it is done in the local existence results in \cite{kargin2011free} as well) and will denote them by e.g. $a(X(t)).$


\section{Free CIR equations}\label{freecirmain}

As mentioned in the introduction it is known in the (commutative) scalar-valued case that the Feller condition \eqref{feller} 
ensures for the CIR equation 
\begin{equation*}
\text{d}x(t)=\left(a-b x(t)\right)\text{d}t+\sigma \sqrt{{x(t)}}\text{d}B(t),\quad x(0)=x_0 >0,
\end{equation*}
for $t\in[0,\infty[, $ and $a,b,\sigma>0$ 
a global positive solution, where $(B(t))_{t\geq0}$ is a (classical) Brownian motion.  

Therefore it is natural to ask, if the above Feller condition guarantees global existence of a positive solution for a free SDE
\begin{equation}\label{E:Free01}
\text{d}X(t)=\left(a(t)-b X(t)\right)\text{d}t+\frac{\sigma(t)}{2} \sqrt{X(t)}\text{d}W(t) + \text{d}W(t) \sqrt{X(t)} \frac{\sigma(t)}{2} ,\quad X(0)=X_0\in\mathcal{A}_+,
\end{equation}
for $a(t),\sigma(t)\in\mathcal{A}_+ , b>0$ and $t\in[0,\infty[,$ where $(W(t))_{t\geq0}$ is a free Brownian motion. From here on we will refer to the (classical) Brownian motion by $(B(t))_{t\geq0}$ and to the free Brownian motion by $(W(t))_{t\geq0}.$ Note that in this paper the above differential notation is always to be understood as an integral equation (see \cite{kargin2011free}). Furthermore, mark that the particular form of \eqref{E:Free01} ensures that $X(t)$ keeps self-adjoint for all $t\in[0,\infty[.$

We therefore state our main theorem:
\begin{theorem}\label{T:CIRfreeBM2}
Let $X_0\in\mathcal{A}_{+}$ be given. Let $a,\sigma:[0,\infty[\rightarrow\left<X_0,\text{id}\right>_+$ be continuous and $b>0$ be a constant such that \eqref{feller} holds.
Then the free SDE
\begin{equation*}
\text{d}X(t)=\left(a(t)-b X(t)\right)\text{d}t+\frac{\sigma(t)}{2} \sqrt{X(t)}\text{d}W(t) + \text{d}W(t) \sqrt{X(t)} \frac{\sigma(t)}{2} ,\quad X(0)=X_0\in\mathcal{A}_+,
\end{equation*}
for $t\in[0,\infty[, $ has a global solution $X\in C([0,\infty[,\mathcal A_+)$.
\end{theorem}

Again, note that this equation is to be understood as the integral equation
\begin{equation*}
X(t)=\int_0^t \left(a(s)-b X(s)\right)\text{d}s+ \int_0^t \frac{\sigma(s)}{2} \sqrt{X(t)}\text{d}W(s) + \int_0^t  \text{d}W(s) \sqrt{X(s)} \frac{\sigma(s)}{2}.
\end{equation*}

As mentioned before, due to the non-commutativity of $X(t),$ the classical free CIR equation 
\begin{equation}\label{E:Free02}
\text{d}X(t)=\left(a(t)-b X(t)\right)\text{d}t+\sqrt[\leftroot{0}\uproot{1}4]{X(t)}\sqrt{\sigma(t)}\text{d}W(t)\sqrt{\sigma(t)}\sqrt[\leftroot{0}\uproot{1}4]{X(t)},\quad X(0)=X_0\in\mathcal{A}_+,
\end{equation}
for $a(t),\sigma(t)\in\mathcal{A}_+ , b>0$ and $t\in[0,\infty[,$ may differ from the non-classical free CIR equation \eqref{E:Free01} but is shown to have a global positive solution, as well, which will be formulated as a corollary at the end of the proof of the main theorem.

We start showing existence and uniqueness of these free SDEs by introducing a SDE, referred to as ``square-root process'', with a simple additive Brownian motion term of the form
$$\text{d}{\overline{U}}(t)=\left(\frac{1}{2}\left(a-\frac{\sigma^2}{4}\right){\overline{U}}^{-1}(t)-\frac{b}{2} {\overline{U}}(t)\right)\text{d}t+\frac{\sigma}{2}\text{d}B(t),\quad \overline{U}(0)=U_0>0,$$  
for $t\in[0,\infty[,{}a,b,\sigma>0$ and $\varphi$ transforming it into a free SDE (driven by a free Brownian motion). Note that the name ``square-root process'' is chosen due to the fact that squaring this process results in the CIR equation in the classical scalar-valued case.

Since the existence of a positive solution to a scalar-valued SDE of the form above and the classical CIR equation are equivalent by the classical It\^{o} lemma, we know that such an equation has a positive solution as long as the Feller condition is satisfied. We will transform this connection first into the setting of (commutative) function spaces (see Theorem~\ref{teozwei}), followed up by the general (non-commutative) von Neumann algebra-valued case (see Lemma~\ref{teodrei}). Note that the driving process is still a classical Brownian motion. 

Since we want to deduce the existence of a positive solution to the free CIR equation(s) we need to rearrange the square-root process in such a way that self-adjointness is ensured. We will furthermore distinguish between a classical Brownian motion- and a free Brownian motion-driven operator-valued version. Note that in order to establish the isometry in the proof of the main theorem (Theorem~\ref{teofunf}) those version have to differ by more than just the Brownian motions, due to their different properties. We therefore define the following two equations:
\begin{eqnarray}\label{vquerupdate}
\text{d}{\overline{V}}(t)&=&\left(\frac{1}{2}\left(a-\frac{\varphi(\sigma)\sigma + \sigma^2}{8}\right){\overline{V}}^{-1}(t)  - \frac{b}{2} \overline{V}(t)\right) \text{d}t+\frac{1}{2}\sqrt{\frac{\varphi(\sigma^2)}{2} + \frac{\varphi(\sigma)^2}{2}}\text{d}B(t),\nonumber\\ &\quad& \overline{V}(0)=\overline{V}_0 = U_0 >0.
\end{eqnarray}
and
\begin{eqnarray}\label{vupdate}
\text{d}V(t)&=&\left(\left(a-\frac{\varphi(\sigma)\sigma + \sigma^2}{8}\right)\frac{1}{4}V^{-1}(t) + \frac{1}{4}V^{-1}(t) \left(a-\frac{\varphi(\sigma)\sigma + \sigma^2}{8}\right) - \frac{b}{2} \overline{V}(t)\right)\text{d}t+\frac{\sigma}{4}\text{d}W(t)\nonumber\\&+& \text{d}W(t) \frac{\sigma}{4},\quad V(0)=V_0=\overline{V}_0 = U_0 >0.
\end{eqnarray}
Note that both versions start at the same point $\overline{V}_0 = V_0.$ Furthermore we want to point out the fact that since the algebra $\left<V_0 , \text{id}\right>$ is commutative, multiplication in \eqref{vquerupdate} is commutative, which is not the case for equation \eqref{vupdate}.
 
The final and most elaborate part (Theorem~\ref{teofunf}) consists of changing the driving process to a free Brownian motion. We will do this by showing that the solutions to the von Neumann algebra-valued SDE driven by a classical and the one driven by a free Brownian motion are $L_2$-isometric for $t\geq0.$ 
Starting with an element $V_0 = \overline{V}_{0} >0$ in the von Neumann algebra $\mathcal{A},$ the main result in \cite{kargin2011free} ensures the local existence of a solution $V(t)$ on an interval $[0,T[$ for a $T>0.$ For a partition $0=t_0 < \cdots < t_n = T$ of the common existence interval we show the isometry for each interval $[t_{i}, t_{i+1}[$ by an induction argument. A crucial part in establishing the isometry are the similar properties of the free and the classical Brownian motions. In particular the fact that the first and the second moments, respectively, of their corresponding increments coincide, is used to show the isometry at $t_i.$ Furthermore, the specific form of these equations permits coinciding ordinary differential equations (ODEs) for the corresponding traces, which, for isometric starting values $V(t_i)$ and $\overline{V}(t_i)$, have a unique solution. This construction establishes the isometry for the interval $[t_{i},t_{i+1}[.$ The final part in the proof shows the invertability of $V(T),$ making the solution global.
Finally the free It\^{o} lemma gives us the existence of a global solution to our free CIR equation (under Feller). Note that this proof can be mimicked for similar free SDEs but we restrain from including such results since they do not fit the scope of this paper.\\


Given a positive $U_0$ we can select a special probability space to transfer the SDE into a usual vector-valued SDE. Using functional calculus resp. the spectral theorem we have an isometric homomorphism
\begin{equation*}
T:\mathcal{B}(\sigma(U_0))\longrightarrow \left< U_0 , \text{id}\right>,
\end{equation*}
where $\left< U_0,\text{id}\right>$ is the von Neumann algebra generated by $U_0$ (and the identity) and $\mathcal{B}(\sigma(U_0))$ is the function space of bounded, measurable functions on $\sigma(U_0), $ the spectrum of $U_0.$  
If $\varphi$ is a unital, faithful trace, then consider $\mathbb E_{\mathbb P_\varphi}=T^*(\varphi)$ ($\mathbb{P}_{\varphi}$ the spectral measure of $U_0$) with the identity
$$\mathbb E_{\mathbb P_\varphi}(g)=\int_{\sigma(U_0)}g\text{d}\mathbb P_{\varphi}=\varphi( T(g) )\text{ for all }g\in \mathcal{B}(\sigma(U_0)).$$ 
Note the implicit definition of the probability measure $\mathbb{P}_{\varphi}$ obtained by choosing $g$ as the identity on $L$ for $L\subseteq \sigma(U_0).$
We may omit the subscript ``$\varphi$'' 
when the context is clear.

%


We fix a filtered probability space $(\Omega,\mathcal{F},(\mathcal{F}_t)_{t\geq0},\mathbb{P})$ and start with a vector-valued version for the existence of a local solution of a classical SDE. For this let $(B(t))_{t\geq0}$ be a classical Brownian motion on $(\Omega,\mathcal{F},(\mathcal{F}_t)_{t\geq0},\mathbb{P}).$
First, we state a result which seems folklore, but is not cited explicitly.

\begin{proposition}\label{teoeins}
Let $E$ be a Banach space. Let $a:E \rightarrow E$ and $b_i : E \rightarrow E$ for $i=1,\dots, m$ be continuous and locally Lipschitz. For $U_0\in E $ 
there exists a $T>0$ and a unique continuous solution $U:[0,T[\rightarrow E,$ 
such that $U(0)=U_0$, $U\in C([0,T[, E)$ and for all $t\in[0,T[$ 
\begin{equation*}
U(t) = U_0 + \int_0^t a(U(s))\text{d}s+\displaystyle\sum_{i=1}^m \int_0^t b_i (U(s))\text{d}B(s).
\end{equation*}
\end{proposition}


\begin{proof}
The proof follows a standard argument via Picard approximations and the Banach fixed-point theorem. Applying the Hahn-Banach theorem we use an element $e^{*}\in E^{*}$ in the dual of $E$ to transform our approximations into the real setting in order to make use of the (classical) Burkholder-Gundy inequality. The details are analogous to the proof of \cite[Theorem~3.1.]{kargin2011free}. $\hfill\qed$
\end{proof}

Having established the existence of a local solution we will show that the classical Feller condition ensures a positive solution in the case of the commutative von Neumann algebra $C(K)$ of continuous functions of a compact Hausdorff space $K.$ We further denote the positive cone by 
$$C(K)_+ = \{f\in C(K)\;|\;f(z)>0,\text{ for all }z\in K\}.$$ 

We note that using the It\^{o} formula, under the Feller condition the scalar-valued square-root process enjoys a global solution provided the initial condition is positive. We state a vector-valued extension.

\begin{theorem}\label{teozwei}
Let $K\subset ]0,\infty[$ be compact. Let $a,b,\sigma:[0,\infty[\rightarrow C(K)_+$ be continuous, such that \eqref{feller} holds and let $\hat{U}_0\in C(K)_+ .$ Then the SDE 
\begin{equation}\label{Uneu}
\text{d}\hat{U}(t)=\left(\frac{1}{2}\left(a(t)-\frac{\sigma^2 (t)}{4}\right)\hat{U}^{-1} (t) - \frac{b(t)}{2}\hat{U}(t)\right)\text{d}t+\frac{\sigma(t)}{2}\one_K \text{d}B(t),\quad \hat{U}(0)=\hat{U}_0 ,
\end{equation} 
for $t\in[0,\infty[,$ has a global solution $U\in C([0, \infty[,L_2(\mathbb P_\varphi, C(K)_{+})).$ $\one_K$ denotes the indicator function on $K$ and $L_2(\mathbb P_\varphi, C(K)_{+})$ denotes the $C(K)_{+}$-valued square-integrable functions w.r.t. to $\mathbb{P}_{\varphi}$.
\end{theorem}
\begin{proof}
Consider a classical Brownian motion $(B(t))_{t\geq0}$ on the probability space $(\Omega,\mathcal{F},\mathbb{P}_{\varphi}).$ Then using point mass measures $\delta_k$ for $k\in K$ and the classical Feller condition prove the global existence of 

\begin{equation*}
\text{d}\hat{U}_{k}(t)=\left(\frac{1}{2}\left(a_{k}(t)-\frac{\sigma_k^2 (t)}{4}\right)\hat{U}_{k}^{-1}(t)-\frac{b_{k}(t)}{2} \hat{U}_k (t)\right)\text{d}t+\frac{\sigma_{k}(t)}{2}\text{d}B(t),\quad\hat{U}_{k}(0)=\hat{U}_{0,k}>0,
\end{equation*}
for $t\in[0,\infty[. $

Using a countable dense subset $\tilde{K}\subset K$ and the point mass functional we show that the paths stay positive except at a $\mathbb P$-zero set $N:=\bigcup_{k\in\tilde{K}} N_k$. So for all $\omega\in\Omega\setminus N$ the paths stay positive on $[0,\infty[$ for all $k.$ $\hfill\qed$
\end{proof}



%
%
\begin{corollary}\label{C:Cir}
Let $K,a,b,\sigma$ be given as in Theorem~\ref{teozwei}. Then the generalized CIR equation 
$$\text{d} \hat{X}(t) = \left(a(t) - b(t) \hat{X}(t)\right)\text{d}t + \sigma(t) \sqrt{\hat{X}(t)}\text{d}B(t),\quad \hat{X}(0) = \hat{X}_0\in C(K)_+,$$ 
for $t\in[0, \infty[, $ has a global solution.
\end{corollary}

\begin{proof}
The proof is immediate by It\^{o}'s lemma with $\text{d}\hat{X}(t) = \text{d}\hat{U}(t)^2$. $\hfill\qed$
\end{proof}

The next step is to transform the square-root process into the setting of a non-commutative von Neumann algebra $\mathcal{A}$ but still with a classical Brownian motion as driving process.


\begin{lemma}\label{teodrei}
Let $U_0 \in \mathcal{A}_+$ and let $a,\sigma,b : [0,\infty[\rightarrow\left<U_0,\text{id}\right>_{+}$ be continuous such that $\eqref{feller}$ holds for all $t\in[0,\infty[.$ 
Then the SDE 
\begin{equation}\label{numberthree}
\text{d}\overline{U}(t) = \left(\frac{1}{2}\left(a(t) - \frac{\sigma^2 (t)}{4}\right)\overline{U}^{-1} (t) - \frac{b(t)}{2}\overline{U}(t)\right)\text{d}t + \frac{\sigma (t)}{2} \text{d}B(t),\quad \overline{U}(0) =  U_0,
\end{equation}
for $t\in[0,\infty[, $ has a global solution in $\overline{U}\in C([0,\infty[, L_2(\mathbb P_\varphi,\mathcal{A}_+)).$ Note that $\text{id}$ is the unit in the corresponding von Neumann algebra. 
\end{lemma}

\begin{proof}
Let $\hat{U}(t)$ be a global solution in $C(K)_+$ by Theorem~\ref{teozwei}. By the functional calculus 
we see that $T(\hat{U}(t))$ is a positive solution to \eqref{numberthree} under the Feller condition, where $T:C(K)\longrightarrow  \left<U_0,id\right>$ is the functional calculus mapping for the self-adjoint element $U_0$. 
In particular, we have $\overline{U}(t) \in \left<U_0,id\right>_{+}$ for all $t\in[0, \infty[$. 
$\hfill\qed$ 
\end{proof}
We formulate two corollaries. Firstly, applying the It\^{o}'s lemma to $\text{d}\overline{U}^2$ gives us:
\begin{corollary}\label{C:Cir2}
Let $a(t),b(t),\sigma(t)$ be given as in Lemma~\ref{teodrei}. Then the (operator-valued) generalized CIR equation
\begin{equation*}
\text{d}\overline{X}(t)=\left( a(t) -b(t) \overline{X}(t)\right)\text{d}t + \sigma(t)\sqrt{\overline{X}(t)}\text{d}B(t),\quad\overline{X}(0)=\overline{X}_0 \in\mathcal{A}_+,
\end{equation*}
for $t\in[0,\infty[, $ has a global solution $\overline{X}\in C([0,\infty[,L_2(\mathbb P_\varphi,\mathcal A_+))$.
\end{corollary}
Secondly, since the classical Feller condition \eqref{feller} implies the inequality $2a(t) \geq \varphi(\sigma(t))\sigma(t)$ by the Jensen's inequality and therefore $2a(t) \geq \frac{\varphi(\sigma(t))\sigma(t)+\sigma(t)^2}{2}$ for all $t\in[0, \infty[,$ we get 
\begin{corollary}\label{extendedsquareroot}
Let $V_0 \in \mathcal{A}_+$ and let $a,\sigma,b : [0,\infty[\rightarrow\left<V_0,\text{id}\right>_{+}$ be continuous such that $\eqref{feller}$ holds for all $t\in[0,\infty[.$ 
Then the SDE 
\begin{eqnarray}\label{numberthreebeta}
\text{d}{\overline{V}}(t)&=&\left(\frac{1}{2}\left(a(t)-\frac{\varphi(\sigma(t))\sigma(t) + \sigma(t)^2}{8}\right){\overline{V}}^{-1}(t)  - \frac{b(t)}{2} \overline{V}(t) \right)\text{d}t+\frac{1}{2}\sqrt{\frac{\varphi(\sigma(t)^2)}{2} + \frac{\varphi(\sigma(t))^2}{2}}\text{d}B(t), \nonumber\\ &\quad & \overline{V}(0)=\overline{V}_0 = V_0=U_0 >0.
\end{eqnarray}
for $t\in[0,\infty[, $ has a global solution in $V\in C([0,\infty[, L_2(\mathbb P_\varphi,\mathcal{A}_+)).$
\end{corollary}

In the above Corollary~\ref{extendedsquareroot} we restricted the coefficient functions to the von Neumann subalgebra $\left< V_0,\text{id}\right>_{+}.$  
We now follow up with our main Theorem~\ref{teofunf}, where we show the existence of our solution in the case of a free Brownian motion. 
For this setting we will have to restrict $b$ to the case of a scalar $b>0.$ 

\begin{remark}
\normalfont
\begin{enumerate}
\item
The last step we need for the existence of a positive solution in the context of free probability, is to change the driving process in the general von Neumann algebra-valued SDE from a classical Brownian motion to a free Brownian motion.  
Corollary~\ref{extendedsquareroot} gives the global existence of a so-called reference solution. Assuming a solution $V(t)$ of the free analogue of \eqref{numberthreebeta} on a maximal interval $[0,T[$ according to \cite{kargin2011free}, we will prove in Theorem~\ref{teofunf} the $L_2$-isometry of $\overline{V}(t)$ and $V(t)$ (resp. the $L_1$-isometry of $p  \overline{V} (t)^2 p$ and $p V (t)^2 p$ for certain projections $p$). With the help of the $L_2$-isometry and the global existence of the reference solution we may define an element $V(T) \in L_2(\varphi)$, which turns out to be an element in $\mathcal{A}_{+}$, using again the reference solution, the $L_1$- and $L_2$-isometries and the proceeding Proposition~\ref{P:P1}. Again using \cite{kargin2011free} we may now extend $V$ beyond $T$ and the global existence is proven.

\item
For a more convenient notation we may write $\varphi$ instead of $\mathbb{E}_{\mathbb{P}_{\varphi}}$ in the sequel. 
 No confusion will arise at any point. 
\item 
For Lipschitz-continuous and $\mathcal A$-valued functions $a,b,c$ we call a process
$$V(t)=V_0+\int_0^t a(V(s))ds+\int_0^tb(V(s))dW(s)c(V(s)),$$
for $t\in[0,T[, $ a free It\^{o} process. The existence is guaranteed by the main result due to \cite{kargin2011free}. 
\end{enumerate}
\end{remark}

Before stating the main result on the existence of a global solution, we introduce the notion of a reference process and subsequently formulate a result, which may be of independent interest. 

\begin{definition}
We say that a process $X\in C([0,T[,\mathcal{A}_{+})\cap C([0,T], L_1 (\varphi))$ allows reference processes, if for all $\hat{T} \in [0,T[$ there exists a process 
$\overline{X} \in C([\hat{T} , T] , \mathcal{A}_{+})$ s.t.
\begin{eqnarray}\label{einsG}
\lVert p X(t) p \rVert_1 &=& \lVert p \overline{X} (t) p \rVert_1 \text{ for }p\in \langle X(\hat{T}) , \text{id}\rangle_+, \text{ resp. }p\text{ free of }X(t) \text{ for all }t\in[\hat{T} , T]\nonumber\\
X(\hat{T}) &=& \overline{X}(\hat{T})\\
\lVert p (X (T) - X (t)) p \rVert_1 &=& \lVert p(\overline{X} (T) - \overline{X} (t))p\rVert_1\text{ for all }t\in[\hat{T},T] \text{ and }p \text{ free of } (X (T) - X (t)).\nonumber
\end{eqnarray}
\end{definition}


\begin{proposition}\label{P:P1} Let $T>0$ and let $X\in C([0,T[,\mathcal A_{+})\cap C([0,T], L_1(\varphi)_{+})$ be a free It\^{o} process, which admits reference processes. 
\begin{enumerate}
\item Then $X(T)\in \mathcal A$.
\item If $\overline{X}(t)$ is invertible for all reference processes $\overline{X}$ and $t\in [\hat{T},T],$ and if $X(t)$ is invertible for all $t\in [0,T[$ and $X\in C([0,T],\mathcal{A})$, then $X(T)$ is also invertible.
\end{enumerate} 
\color{black}
\end{proposition}  
\begin{proof}
\begin{enumerate}
\item
We suppose that $X(T)\not\in\mathcal A$. Then, since $X(T)\in L_1(\varphi)_{+}$, we find a sequence of projections $(p_n)$, such that
\begin{equation}\label{E:step22}
\frac1{\varphi(p_n)}\|p_nX (T) p_n\|_1\ge n.
 \end{equation}


We may write $X(T) = X(\hat{T})+(X(T)-X(\hat{T}))$ for all $\hat{T}\in ]0,T[.$ The element $X(T)$  is generated by the projections in the von Neumann algebras 
$\mathcal{A}_{\hat{T}}$ and $\mathcal{A}_T,$ which are generated by $\{W(t)\;|\;t\leq\hat{T}\}\cup\{X_0,\text{id}\}$ and $X(T)-X(\hat{T}),$ respectively. 
Therefore $\mathcal{A}_{\hat{T}}$ and $\mathcal{A}_T$ are free.
Consider two cases:\\

\begin{itemize}
\item [a)] 

Suppose there is a $\hat{T}\in[0,T[$ s.t. $p_n \in \mathcal{A}_{\hat{T}}$ for infinitely many $n\in\mathbb{N}$ (for simplicity we assume for all $n\in\mathbb{N}$). 
In this case we consider the reference process with a new starting value $X(\hat{T})$ instead of $X(0)$ and instead of the interval $[0,T]$ the interval $[\hat{T},T].$ 
Then for all $n\in\mathbb{N}$ the projection $p_n$ is free of increments of the free Brownian motion $W(t) - W(s)$ for $\hat{T}\leq s<t\leq T.$ 
\item [b)] For all $\hat{T} \in [0,T[$ we find $p_n \in\mathcal{A}_{\hat{T}}$ only for finitely many $n\in\mathbb{N}.$ Then we choose the infinitely many $p_n \in\mathcal{A}_T,$ for simplicity all 
$n\in\mathbb{N}.$ Here, all those $p_n$ are free of $\mathcal{A}_{\hat{T}}$ and thus this time we have that $W(t) - W(s)$ for $0\leq s < t \leq\hat{T}$ are free of $p_n.$\\
\end{itemize}

We start with case a) and select this specific $\hat{T}\in [0,T[.$

Since $\overline{X}\in C([\hat{T},T], \mathcal{A}_+),$ we find a $\hat{T}\leq\tilde{T} < T$ s.t. for all $\tilde{T}\leq t \leq T$
\begin{equation*}
\lVert \overline{X} (T) - \overline{X} (\tilde{T})\rVert < \frac{1}{2}
\end{equation*}
and
\begin{equation}\label{dreiG}
\sup_{t\in [0,T]} \lVert \overline{X}(t) \rVert  = M < \infty.
\end{equation}
Since $\overline{X}(\tilde{T}) \in \mathcal{A},$ we get 
\begin{equation*}
\frac{1}{\varphi(p_n)} \lVert p_n \overline{X} (\tilde{T}) p_n \rVert_1 \leq \lVert \overline{X} (\tilde{T})\rVert < \infty.
\end{equation*}
Then we have 
\begin{equation}\label{erstb}
\varphi(p_n\overline{X} (t) p_n) = \varphi(p_n X (t) p_n)\text{ for }t\in[\tilde{T},T].
\end{equation}
Consequently
\begin{eqnarray*}
\infty>\|\overline{X} (\tilde{T})\|&\ge&\frac{1}{\varphi(p_n)}\lVert p_n \overline{X} (\tilde{T}) p_n\rVert_1\\ &=& \frac{1}{\varphi(p_n)}\lVert p_n (\overline{X} (T)+(\overline{X} (\tilde{T}) - \overline{X} (T))) p_n\rVert_1\\
&\stackrel{(\ref{erstb})}{\geq}&\left\lvert \frac{1}{\varphi(p_n)}\lVert p_n X (T) p_n\rVert_1 - 
\frac{1}{\varphi(p_n)}\lVert p_n (\overline{X} (T) - \overline{X} (\tilde{T})) p_n\rVert_1 \right\rvert\\
&>&n-1,
\end{eqnarray*}
a contradiction.\\

Now we suppose case b). We select a $\hat{T},$ s.t. $\lVert \overline{X}(t) - \overline{X}(\hat{t})\rVert<1$ for all $t,\hat{t}\in [\hat{T},T].$\\
 Then we may find a subsequence of projections $(p_n)$ (we denote it with the same index), s.t. $p_n \in \langle W_{t_n} - W_{s_n}|\hat{T} \leq s_n < t_n < T\rangle_{+}.$ Then $p_n$ is free of $X(s_n)$ and free of $X (T) - X (t_n) .$ Going to another subsequence we may assume $\tilde{T} \leq s_n < t_n \leq s_{n+1} < t_{n+1} < T.$ Now we estimate
\begin{eqnarray}\label{sechsG}
\frac{1}{\varphi(p_n)} \lVert p_n X (s_n) p_n \rVert_1 &+& \frac{1}{\varphi(p_n)}\lVert p_n (X (T) - X (t_n) )p_n \rVert_1 \nonumber \\
= \frac{1}{\varphi(p_n)} \lVert p_n \overline{X}(s_n) p_n \rVert_1 &+& \frac{1}{\varphi(p_n)}\lVert p_n (\overline{X} (T) - \overline{X} (t_n))p_n \rVert_1 \leq M+1 
\end{eqnarray}
according to \eqref{einsG}.

Since we have for $t>t_n$ that $\frac{1}{\varphi(p_n)} \Vert p_n (X (t) - X (t_n)) p_n \rVert_1 = \frac{1}{\varphi(p_n)} \Vert p_n (\overline{X} (t) - \overline{X} (t_n)) p_n \rVert_1$ and $\overline{X}\in C([0,T], \mathcal{A}),$ we see that 
\begin{equation*}
\lim_{n\rightarrow\infty} \frac{1}{\varphi(p_n)} \lVert p_n (X (T) - X (t_n)) p_n\rVert_1 = 0,
\end{equation*}
which implies there exists an $N\in\mathbb{N}$ s.t. for all $n \geq N$ we have
\begin{equation}\label{siebenG}
\frac{1}{\varphi(p_n)} \lVert p_n X (t_n) p_n \rVert_1 \geq 2M + 1,
\end{equation}
Now, we see with \eqref{sechsG} and the freeness of $p_n$ that 
\begin{eqnarray*}
1>\frac1{\varphi(p_n)}\|p_n (\overline{X}(s_{n+1})-\overline{X}(t_n))p_n\|_1&=&
\frac1{\varphi(p_n)}\|p_n (X(s_{n+1})-X(t_n))p_n\|_1 \\
&\ge & \left|\frac1{\varphi(p_n)} \|p_n X(t_n)p_n\|_1 -\frac1{\varphi(p_n)}\|p_n X(s_{n+1})p_n\|_1 \right|\\
 &\ge & 2M+1-M=M+1>1,
\end{eqnarray*}
 a contradiction.\\
\item  Suppose $X (T)$ is not invertible.  We find a sequence of projections $(p_n)$ in the von Neumann algebra generated by the self-adjoint element $X(T)$ such hat
\begin{equation}\label{erst1}
\frac{1}{\varphi(p_n)}\lVert p_n X (T) p_n\rVert_1\longrightarrow0,\text{ }n\rightarrow\infty.
\end{equation}
In case a) we find a $\hat{T} < T$ and the sequence $(p_n) \in \mathcal{A}_{\hat{T}}.$ 

We know that $\overline{X}(T)$ is invertible and therefore 
$$0<\frac{1}{\lVert (\overline{X} (T))^{-1}\rVert}\leq\frac{1}{\varphi(p_n)}\lVert p_n \overline{X} (T) p_n\rVert_1.$$
Again we consider for all reference processes the two cases a) and b) as in 1.\\

In this case we select the interval $[\hat{T} , T]$ instead of $[0,T]$ and $X(\hat{T})$ as starting value. Then according to \eqref{einsG} in the definition of a reference process $\overline{X},$ we have 

\begin{equation}\label{neunG}
\lVert p_n X (t) p_n \rVert_1 = \lVert p_n \overline{X} (t) p_n \rVert_1 \text{ for }t\in [\hat{T} , T].
\end{equation}

 Since both processes are continuous mappings $$\overline{X},X:[0,T]\longrightarrow\mathcal{A},$$
i.e. $\overline{X},X\in C([0,T],\mathcal{A}),$ we can find a $\tilde{T}\in[\hat{T}, T[,$ s.t. 
\begin{equation*}
\left\lvert \frac{1}{\lVert(\overline{X} (T))^{-1}\rVert}-\frac{1}{\lVert(\overline{X} (\tilde{T}))^{-1}\rVert}\right\rvert<\frac{1}{4\lVert(\overline{X} (T))^{-1}\rVert}
\end{equation*}
and 
\begin{equation*}
\lVert X (T)-X (\tilde{T}) \rVert<\frac{1}{4\lVert(\overline{X} (T))^{-1}\rVert}.
\end{equation*}

Thus, for all $n\in\mathbb{N}:$
\begin{equation*}
\beta_n = \frac{1}{\varphi(p_n)}\lVert p_n (X (T)-X (\tilde{T}))p_n\rVert_1\leq\frac{1}{\varphi(p_n)} \varphi(p_n) \lVert X (T)-X (\tilde{T})\rVert< \frac{1}{4\lVert(\overline{X} (T))^{-1}\rVert}.
\end{equation*}
By \eqref{neunG} 
\begin{equation}\label{dreizehnG}
\varphi(p_n \overline{X} (t) p_n) = \varphi(p_n X (t) p_n)\text{ for }t\in[\tilde{T}, T].
\end{equation}
In case b) first we may find $\tilde{T}\in ]0,T[$ as above (but independent of any $\hat{T}$). By case b) equation \eqref{dreizehnG} is fulfilled for $t=\tilde{T}.$\\

Consequently, it holds in both cases that
\begin{eqnarray*}
\frac{1}{\varphi(p_n)}\lVert p_n X (T) p_n\rVert_1 &=& \frac{1}{\varphi(p_n)}\lVert p_n (X (\tilde{T})+(X (T) - X (\tilde{T}))) p_n\rVert_1\\
&\stackrel{(\ref{dreizehnG})}{\geq}&\left\lvert \frac{1}{\varphi(p_n)}\lVert p_n \overline{X} (\tilde{T}) p_n\rVert_1 - 
\frac{1}{\varphi(p_n)}\lVert p_n (X (T) - X (\tilde{T})) p_n\rVert_1 \right\rvert\\
&>&\frac{3}{4\lVert (\overline{X} (T))^{-1}\rVert}-\beta_n 
\geq\frac{1}{2\lVert(\overline{X} (T))^{-1}\rVert}>0.
\end{eqnarray*}

This is a contradiction to the assumption (\ref{erst1}) and hence $X(T)$ is invertible. $\hfill\qed$

\end{enumerate}
\end{proof}

\color{black}
\begin{theorem}\label{teofunf}
Let $V_0\in\mathcal{A}_{+}$ be given. Let $a,\sigma:[0,\infty[\rightarrow\left<V_0,\text{id}\right>_+$ be continuous and $b>0$ be a constant such that \eqref{feller} holds. Then the free SDE 


\begin{eqnarray}\label{globalo}
\text{d}{V}(t)&=&\biggl(\left(a(t)-\frac{\varphi(\sigma(t))\sigma(t) + \sigma(t)^2}{8}\right)\frac{1}{4}{V}^{-1}(t) + \frac{1}{4}{V}^{-1}(t) \left(a(t)-\frac{\varphi(\sigma(t))\sigma(t) + \sigma(t)^2}{8}\right)\nonumber\\ &-& \frac{b}{2} V(t)\biggr)\text{d}t+\frac{\sigma(t)}{4}\text{d}W(t) + \text{d}W(t) \frac{\sigma(t)}{4},\quad V(0)=V_0 >0,
\end{eqnarray}

for $t\in[0,\infty[, $ has a global solution  $V\in C([0,\infty[,\mathcal A_+)$.
\end{theorem}

\begin{proof}
1. In the first step we choose a maximal interval $[0,T[,$ where a solution of the equation 
\begin{eqnarray*}
\text{d}{V}(t)&=&\biggl(\left(a(t)-\frac{\varphi(\sigma(t))\sigma(t) + \sigma(t)^2}{8}\right)\frac{1}{4}{V}^{-1}(t) + \frac{1}{4}{V}^{-1}(t) \left(a(t)-\frac{\varphi(\sigma(t))\sigma(t) + \sigma(t)^2 }{8}\right)\nonumber\\ &-& \frac{b}{2} V(t)\biggr)\text{d}t+\frac{\sigma(t)}{4}\text{d}W(t) + \text{d}W(t) \frac{\sigma(t)}{4}, \quad V(0)=V_0 >0,
\end{eqnarray*}
for $t\in[0,T[, $ exists according to \cite{kargin2011free}. By Corollary~\ref{extendedsquareroot} we know that the solution for 
\begin{eqnarray*}
\text{d}{\overline{V}}(t)&=&\left(\frac{1}{2}\left(a(t)-\frac{\varphi(\sigma(t))\sigma(t) + \sigma(t)^2}{8}\right){\overline{V}}^{-1}(t)  - \frac{b(t)}{2} \overline{V}(t) \right)\text{d}t+\frac{1}{2}\sqrt{\frac{\varphi(\sigma(t)^2)}{2} + \frac{\varphi(\sigma(t))^2}{2}}\text{d}B(t),\\ &\quad & \overline{V}(0)=V_0 >0
\end{eqnarray*}
for $t\in[0,\infty[, $ exists globally. First we prove the following three isometries.
\begin{eqnarray}\label{isos}
 \lVert V(t)\rVert_2 &=& \lVert \overline{V}(t)\rVert_2\text{ for }t\in[0,T[\nonumber\\
 \lVert p V (t)^2 p\rVert_1 &=& \lVert p\overline{V} (t)^2 p\rVert_1\text{ for }t\in[0,T[\text{ and }p\in\left<V_0,\text{id}\right>_{+}\\  
 \lVert p V (t)^2 p\rVert_1 &=& \lVert p\overline{V} (t)^2 p\rVert_1\text{ for }t\in[0,T[\text{ and }p\text{ free of }V(t)\nonumber
\end{eqnarray}
The third isometry directly follows from the first by freeness.

Note that the norms on the left hand side of \eqref{isos} are to be understood in the sense of 
$$ \lVert V(t)\rVert_2^2 = \lVert V(t)^2\rVert_1$$
and the norms on the right hand side of \eqref{isos} in the sense of 
$$ \lVert \overline{V}(t)\rVert_2^2 =  \lVert \overline{V}(t)\rVert_{L_2 (\mathbb{P},L_2 (\varphi))}^2 = \lVert \overline{V}(t)^2\rVert_{L_1 (\mathbb{P},L_1 (\varphi))} = \lVert \overline{V}(t)^2\rVert_1 .$$
The context will always make it clear which norms are meant. No confusion will arise at any point.

We approximate the solutions  on  $[t_0 , T[$
and thus, select a partition $\mathcal Z$ of the interval $[t_0,T]$ namely $0=t_0 < t_1<\dots<t_n=T.$ We will omit the variable ``$t$'' in the expression of $a$ and $\sigma.$ The proof for the isometry 
$$\lVert V(t)\rVert_2 = \lVert \overline{V}(t)\rVert_2\text{ for }t\in[0,T[$$ is basically a complete induction. We will show it for $[t_0 , t_1]$ and then the step from $[t_0 , t_1]$ to $[t_1,t_2]$ to illustrate the procedure. However, note that $V(t_n) = V(T)$ is undefined at this point, which is why the last step only shows the isometry for the interval $[t_{n-1},t_{n}[.$
 We start with the interval $[t_0,t_1[$. For $t\in[t_0,t_1[$ we have
$${\overline{V}_\mathcal Z (t)}=V_0 + \int_{t_0}^{t} \left(\frac{1}{2}\left(a-\frac{\varphi(\sigma)\sigma + \sigma^2}{8}\right){\overline{V}_\mathcal Z^{-1} (s)}  - \frac{b}{2} \overline{V}_\mathcal Z(s)\right)\text{d}s.$$
and
$${V_\mathcal Z (t)}=V_0 + \int_{t_0}^{t} \left(\left(a-\frac{\varphi(\sigma)\sigma + \sigma^2}{8}\right)\frac{1}{4}{V_\mathcal Z^{-1} (s)} + \frac{1}{4}{V_\mathcal Z^{-1} (s)} \left(a-\frac{\varphi(\sigma)\sigma + \sigma^2}{8}\right) - \frac{b}{2} V_\mathcal Z(s)\right)\text{d}s.$$
Thus, we have by an easy approximation by step functions that $\overline{V}_\mathcal Z (t), V_\mathcal Z (t) \in\left<V_0,\text{id}\right>_+ ,$ a commutative algebra, and therefore
$$\varphi(V_{\mathcal Z} (t)^2) = \varphi(\overline{V}_{\mathcal Z} (t)^2) = \lVert \overline{V}_{\mathcal Z}  (t) \rVert_{L_2 (\mathbb{P},L_{2} (\varphi))}^2.$$
By adding the classical Brownian motion term $\frac{1}{2}\sqrt{\frac{\varphi(\sigma^2)}{2} + \frac{\varphi(\sigma)^2}{2}} B(t_1)$ and the free Brownian motion term $\frac{\sigma}{4} W(t_1) + W(t_1)\frac{\sigma}{4},$ respectively 
we define 
\begin{eqnarray*}
\overline{V}_{\mathcal Z}(t_1) &=& \overline{V}_\mathcal Z (t_1-) + \frac{1}{2}\sqrt{\frac{\varphi(\sigma^2)}{2} + \frac{\varphi(\sigma)^2}{2}} B(t_1)\\
V_{\mathcal Z} (t_1) &=& V_\mathcal Z (t_1-) + \frac{\sigma}{4} W(t_1) + W(t_1)\frac{\sigma}{4},
\end{eqnarray*}
where the first one is just a discretization of the global solution of Corollary~\ref{extendedsquareroot} (which serves as our ``reference solution'') evaluated at $t_1$ and $V_\mathcal Z (t_1-) $ and $\overline{V}_\mathcal Z (t_1-),$ respectively is the local solution for $t\in[t_0 , t_1[$. 
Since $\sigma\in\left<V_0,\text{id}\right>_{+}$ is free of $W(t_1)$ and $W(t_1)$ has the properties $ \varphi(W(t_1)) = 0$ and $\varphi(W(t_1)^2) = t_1$  (the same properties hold for $B(t_1)$),  we conclude with 
\cite[p.~829, equation (8)]{kargin2011free} that 
$$\varphi\otimes\mathbb{P}\left(\left(\frac{1}{2}\sqrt{\frac{\varphi(\sigma^2)}{2} + \frac{\varphi(\sigma)^2}{2}} B(t_1)\right)^2\right)=\varphi\left(\left(\frac{\sigma}{4} W(t_1) + W(t_1)\frac{\sigma}{4}\right)^2\right) = \left( \frac{1}{8} \varphi(\sigma)^2 + \frac{1}{8} \varphi(\sigma^2) \right) t_1$$ 
and see by the independence, resp. the freeness that
$$  \varphi(V_{\mathcal Z} (t_1)^2)=\varphi(\overline{V}_{\mathcal Z} (t_1)^2) =\lVert \overline{V}_{\mathcal Z}  (t_1) \rVert_{L_2 (\mathbb{P},L_{2} (\varphi))}^2.$$
Note that when the context is clear we will denote the product measure $\varphi\otimes\mathbb{P}$ for a classical Brownian motion term, e.g. in the above left hand expression $\varphi\otimes\mathbb{P}\left(\left(\frac{1}{2}\sqrt{\frac{\varphi(\sigma^2)}{2} + \frac{\varphi(\sigma)^2}{2}} B(t_1)\right)^2\right),$ by just ``$\varphi$''. No confusion will arise at any point.

On $[t_1, t_2[$ we again consider approximations of our solutions starting in $\overline{V}_\mathcal Z(t_1)$ and $V_\mathcal Z(t_1),$ (not equal but $L_2$-isometric) respectively:
\begin{eqnarray*}
{\overline{V}_\mathcal Z (t)}&=&{\overline{V}_\mathcal Z (t_1)} + \int_{t_1}^{t} \left(\frac{1}{2}\left(a-\frac{\varphi(\sigma)\sigma + \sigma^2}{8}\right){\overline{V}_\mathcal Z^{-1} (s)}  - \frac{b}{2} \overline{V}_\mathcal Z(s)\right)\text{d}s,\\
{V_\mathcal Z (t)}&=&{V_\mathcal Z (t_1)} + \int_{t_1}^{t} \left(\left(a-\frac{\varphi(\sigma)\sigma + \sigma^2}{8}\right)\frac{1}{4}{V_\mathcal Z^{-1} (s)} +\frac{1}{4} {V_\mathcal Z^{-1} (s)} \left(a-\frac{\varphi(\sigma)\sigma + \sigma^2}{8}\right) - \frac{b}{2} V_\mathcal Z(s)\right)\text{d}s
\end{eqnarray*}
Defining $\overline{Y} (t) := \overline{V}_{\mathcal Z} (t)^2 , $ resp. $Y(t) := V_{\mathcal Z} (t)^2$ and applying the trace we get a general ODE of the form $$\frac{\text{d}y}{\text{d}t} = \alpha-by $$ with $\alpha=\varphi(a)-\frac{\varphi(\sigma)^2 + \varphi(\sigma^2)}{8}$ and $y\in\{\varphi(\overline{Y} (t)),\;\varphi(Y(t))\}$ 
with the initial value $$y(t_1)=\varphi(Y(t_1))=\varphi(\overline{Y}(t_1)).$$ 
Since this ODE has a unique solution we get 
$$\lVert \overline{V}_\mathcal Z (t) \rVert_2^2=\varphi(\overline{Y}(t)) = \varphi(Y(t))=\lVert V_\mathcal Z (t) \rVert_2^2\text{ for }t\in[t_1,t_2[.$$
Again we add the classical Brownian motion term $\frac{1}{2}\sqrt{\frac{\varphi(\sigma^2)}{2} + \frac{\varphi(\sigma)^2}{2}} \left(B(t_2) -  B(t_1)\right)$ and the free Brownian motion term $\frac{\sigma}{4} (W(t_2) - W(t_1)) + (W(t_2) - W(t_1))\frac{\sigma}{4},$ respectively and get 
\begin{eqnarray*}
\overline{V}_{\mathcal Z}(t_2) &=& \overline{V}_\mathcal Z (t_2-) + \frac{1}{2}\sqrt{\frac{\varphi(\sigma^2)}{2} + \frac{\varphi(\sigma)^2}{2}} \left(B(t_2) -  B(t_1)\right) \\
V_{\mathcal Z} (t_2) &=& V_\mathcal Z (t_2-) + \frac{\sigma}{4} (W(t_2) - W(t_1)) + (W(t_2) - W(t_1))\frac{\sigma}{4}.
\end{eqnarray*}
As before by the independence, resp. freeness we conclude 
$$\varphi(\overline{V}_{\mathcal Z} (t_2)^2) = \varphi(V_{\mathcal Z} (t_2)^2).$$
We may extend $\overline{V}_\mathcal Z, $ resp. $V_\mathcal Z$ beyond the interval $[0,t_2[$. In the same respect as above we have for both processes extensions on $[0,t_n[$ such that
$$\varphi(\overline{V}_{\mathcal Z} (t)^2) = \varphi(V_{\mathcal Z} (t)^2)\text{ for all  }t\in [0,t_n[.$$
Since these isometries hold by the continuity of the solutions for all partitions of $[0,T],$ we see that $V_\mathcal Z$ converges for $|\mathcal Z|\to 0$, i.e. if the length of the partition converges to $0$:
$$\varphi(\overline{V} (t)^2)=\displaystyle\lim_{|\mathcal Z|\rightarrow 0} \varphi(\overline{V}_{\mathcal Z} (t)^2)=\displaystyle\lim_{|\mathcal Z|\rightarrow 0} \varphi(V_{\mathcal Z} (t)^2) = \varphi(V (t)^2)\text{ for all }t\in[0,T[.$$
We proceed by showing the second isometry of \eqref{isos} and therefore introduce the equation
\begin{eqnarray*}
p V p (t) &=& p V_0 p + \int_0^t \biggl( p \left( a-\frac{\varphi(\sigma)\sigma + \sigma^2}{8}\right)p \frac{1}{4}(p V p)(s)^{-1}  +p \frac{1}{4}(p V p)(s)^{-1} p \left( a-\frac{\varphi(\sigma)\sigma + \sigma^2}{8}\right)p\\ &-& p\frac{b}{2} p V p (s)\biggr)\text{d}s+\int_0^t p \frac{\sigma}{2}\;p\;\text{d}W(s) p + \int_0^t p\;\text{d}W(s) p\;p \frac{\sigma}{2}\;p ,\quad p V p (0) = p V_0 p \in p\mathcal{A}_ {+} p ,
\end{eqnarray*}
with $a,\sigma\in\left<V_0,\text{id}\right>_{+}\text{ and }b>0\text{, all strictly positive,} $ in $p \mathcal{A} p. $ Since $p\in\left< V_0,\text{id}\right>_{+}$ by assumption, we have $pap,\;p \sigma p\in\left< p V_0 p, \text{id}\right>_{+}$ and the proof for the $L_2$-isometry above can be mimicked. 
For this we consider $\text{d}(p V p (t))^2 = 2 p V p(t)\;\text{d}(pVp(t))$ to get a similar general ODE to the case above.
\medskip

2. According to \cite{kargin2011free}, we know that the process $V(t)^2$ exists for $t<T$.
The reference solution $\mathbb{E}_{\varphi}\left(\overline{V}(t)^2\right)$ exists globally with values in $\mathcal A$. Because of the isometry we can define
\begin{eqnarray*}
V(T)=L_2-\lim_{t\to T}V(t)&=&V_0+\int_0^T \biggl(\left(a-\frac{\varphi(\sigma)\sigma + \sigma^2}{8}\right)\frac{1}{4}V^{-1}(t) + \frac{1}{4}V^{-1}(t) \left(a-\frac{\varphi(\sigma)\sigma + \sigma^2}{8}\right)\\&-&\frac{b}{2} V(t)\biggr)\text{d}t
+\frac{\sigma}4W(T) + W(T)\frac{\sigma}4.
\end{eqnarray*}
Using Proposition \ref{P:P1} (1.) we can deduce that $V(T)\in\mathcal A$. Note that in Proposition \ref{P:P1} we choose $\overline{X}(t) = \mathbb{E}_{\varphi}\left(\overline{V}(t)^2\right)$ as the reference process.

3. Step 2 told us that $V(T)\in\mathcal A$.  We want to extend the (unique) solution $V(t)$ beyond $T$. To apply the basic result in \cite{kargin2011free}, we need the invertibility of $V(T).$ This would allow us an extension and the solution is global. Again by Proposition \ref{P:P1} (2.), we see that $V(T)$ is invertible.
 Therefore the solution to \eqref{globalo} is global. $\hfill\qed$
\end{proof}


Having established the existence of the global solution $V(t)$ we return to the proof of the main theorem (Theorem \ref{T:CIRfreeBM2}). We apply the free It\^{o} formula and get the global solution to the free CIR equation:\\

%

Let $Y(t)=V (t)^2$. Then according to the free It\^{o} formula: 
\begin{eqnarray*}
\text{d}Y(t) &=& \bigg( V(t) \left( a- \frac{\varphi(\sigma) \sigma + \sigma^2}{8} \right)\frac{1}{4}V^{-1}(t) + \frac{1}{4}  V^{-1}(t) \left( a- \frac{\varphi(\sigma) \sigma + \sigma^2}{8} \right) V(t)   + \frac{1}{2}\left(a-\frac{\varphi(\sigma) \sigma + \sigma^2}{8}\right)\\ &-& b V^2 (t) +    \frac{\sigma\varphi(\sigma)}{8} + \frac{\sigma^2}{16} + \frac{\varphi(\sigma^2)}{16}\bigg)\text{d}t + V(t)\frac{\sigma}{4}\text{d}W(t) + V(t)\text{d}W(t)\frac{\sigma}{4} + \frac{\sigma}{4}\text{d}W(t)V(t)\\ &+& \text{d}W(t)\frac{\sigma}{4} V(t) ,\quad Y(0) = Y_0. \\
\end{eqnarray*}
Furthermore, we consider another free stochastic differential equation 
\begin{equation*}
\text{d}X(t) = (a- b X(t))\text{d}t + \frac{\sigma}{2}\sqrt{X(t)}\text{d}W(t) + \text{d}W(t)\sqrt{X(t)}\frac{\sigma}{2},\quad X(0) = Y_0.
\end{equation*}

Applying the trace, writing $V(t) = (V(t) - V_0) + V_0$, using the freeness of $\sigma$ to 
$(V(t) - V_0)\text{d}W(t)$ and the fact that the algebra $\left<V_0 , \text{id}\right>$ is commutative, we get
\begin{equation*}
\text{d}\varphi(Y(t)) = \varphi\left(a-b Y(t)\right)\text{d}t + \varphi\left(\frac{\sigma}{2}\sqrt{Y(t)}\text{d}W(t)\right) + \varphi\left(\text{d}W(t)\sqrt{Y(t)}\frac{\sigma}{2}\right),\quad \varphi(Y(0)) = \varphi(Y_0).
\end{equation*}
Similar we have 
\begin{equation*}
\text{d}\varphi(X(t)) = \varphi\left(a-b X(t)\right)\text{d}t + \varphi\left(\frac{\sigma}{2}\sqrt{X(t)}\text{d}W(t)\right) + \varphi\left(\text{d}W(t)\sqrt{X(t)}\frac{\sigma}{2}\right),\quad \varphi(X(0))= \varphi(Y_0).
\end{equation*}

We see that $\lVert Y(t) \Vert_1 = \varphi(Y(t)) = \varphi(X(t)) = \lVert X(t) \rVert_1$ as long as the solution $X$ exists on a maximal interval $[0,T[$ (the same arguments as in \cite{kargin2011free} for a unique solution, since the equation is locally Lipschitz). The positivity of $Y(t),$ respectively $X(t)$ ensures that $\varphi(Y(t)) = \lVert Y(t) \rVert_1,$ respectively $\varphi(X(t)) = \lVert X(t) \rVert_1.$

Since the algebras generated by $Y(t)$ and by $X(t)$ $(t\in [0,T[),$ respectively, coincide, we have 

\begin{equation*}
\lVert p Y(t) p \rVert_1 = \varphi(p Y(t) p) = \varphi(p X(t) p) = \lVert p X(t) p \rVert_1 \text{ for }p\text{ free of }X(t).
\end{equation*}
The same holds for $p\in\left<Y_0 , \text{id}\right>_{+}$ (applying the trace to an element of $p\mathcal{A}p$).

In addition, we have immediately
\begin{equation*}
\varphi(p(Y(T) - Y(t))p) = \varphi(p(X(T) - X(t))p), \text{ for all }p\text{ free of }(X(T) - X(t)) \text{(resp. }(Y(T) - Y(t))\text{)}.
\end{equation*}


Now we apply the notion of reference processes. Proposition \ref{P:P1} shows that $X(T)$ is invertible and with \cite{kargin2011free} we continue (note that in \eqref{sechsG} in the proposition we may use $\left| \varphi(p_n (X (T) - X (t_n))p_n)\right|$ instead of $\varphi(p_n (X (T) - X (t_n))p_n)$). $\hfill\qed$\\


We conclude with the following corollary stating the global existence of the classical free CIR equation.
\begin{corollary}\label{coro4}
Let $X_0\in\mathcal{A}_{+}$ be given. Let $a,\sigma:[0,\infty[\rightarrow\left<X_0,\text{id}\right>_+$ be continuous and $b>0$ be a constant such that \eqref{feller} holds.
Then the free SDE
\begin{equation*}
\text{d}X(t)=\left(a(t)-b X(t)\right)\text{d}t+\sqrt[\leftroot{0}\uproot{1}4]{X(t)}\sqrt{\sigma(t)}\text{d}W(t)\sqrt{\sigma(t)}\sqrt[\leftroot{0}\uproot{1}4]{X(t)}, \quad X(0) = X_0 \in\mathcal{A}_+ ,
\end{equation*}
for $t\in[0,\infty[, $ has a global solution $X\in C([0,\infty[,\mathcal A_+)$.
\end{corollary}


\begin{proof}
The global existence of a positive solution to the non-classical free CIR equation is obtained analogously. 
By writing $\sqrt{X(t)} = \left(\sqrt{X(t)} - \sqrt{X_0}\right) + \sqrt{X_0}$ and additionally employing \cite[p.~17, equation (1.14)]{mingo2017free}, combined with the freeness of $X_0$ and $\text{d}W(t)\left(\sqrt{X(t)} - \sqrt{X_0}\right)$ and the fact that $\varphi\left(\text{d}W(t)\right) =0,$ we get 
\begin{equation*}
\text{d}\varphi(X(t)) = \varphi\left(a(t) - b X(t)\right)\text{d}t + \varphi\left(\sqrt{\sigma(t)}\right)^2 \varphi\left(\text{d}W(t)\sqrt{X(t)}\right). 
\end{equation*}
Again we consider
\begin{equation*}
\text{d}\varphi(Y(t)) = \varphi\left(a(t) - b Y(t)\right)\text{d}t + \varphi\left(\frac{\varphi\left(\sqrt{\sigma(t)}\right)^2}{2} \text{d}W(t)\sqrt{Y(t)} \right) + \varphi\left(\sqrt{Y(t)}\text{d}W(t)\frac{\varphi\left(\sqrt{\sigma(t)}\right)^2}{2}\right).
\end{equation*}
Since $2\varphi(a(t))\geq \varphi(\sqrt{\sigma(t)})^2$ by Jensen's inequality, we get the global existence of the classical free CIR equation.
\end{proof}

We end this manuscript with the following corollary which states the existence of a global solution to an analogue of a free SDE discussed in \cite[p.840ff]{kargin2011free}. The proof is analogous to Corollary \ref{coro4}.

\begin{corollary}
Let $X_0\in\mathcal{A}_{+}$ be given. Let $a:[0,\infty[\rightarrow\left<X_0,\text{id}\right>_+$ be continuous and $b,c>0$ be constants such that \eqref{feller} holds.
Then the free SDE 
$$
\text{d}X(t)=\left(a(t)-b X(t)\right)\text{d}t +  \sqrt{\sigma} \sqrt[\leftroot{0}\uproot{1}4]{X(t)} \text{d}W(t)\sqrt[\leftroot{0}\uproot{1}4]{X(t)}\sqrt{\sigma},\quad X(0) = X_0 \in\mathcal{A}_+
$$
for $t\in[0,\infty[, $ has a global solution $X\in C([0,\infty[,\mathcal A_+)$.
\end{corollary}

\section*{Funding}
This research received no external funding.
\section*{Conflicts of interest}
The authors declare no conflict of interest. 

{\footnotesize
\bibliographystyle{plainnat.bst}
\bibliography{final_version}
}


\end{document}